\documentclass[10pt]{article}
\usepackage{amsmath,amscd,amsthm,amssymb}

 \newtheorem{thm}{Theorem}[section]
 \newtheorem{cor}[thm]{Corollary}
 \newtheorem{lem}[thm]{Lemma}
 
 \theoremstyle{definition}
 \newtheorem{defn}[thm]{Definition}
 \theoremstyle{remark}
 \newtheorem{rem}[thm]{Remark}
 \newtheorem*{ex}{Example}
 \numberwithin{equation}{section}

\def\R{{\mathbb{R}}}
\def\Rn{{\mathbb{R}^n}}

\def\a {\alpha}
\def\i{\infty}
\def\l {\lambda}

\def\L1loc{L_{\Phi}^{\rm loc}(\Rn)}

\def\dual{\,^{^{\complement}}\!}

\usepackage{color}
\newcommand{\tcr}{\textcolor{red}}

\newcommand{\es}{\mathop{\rm ess \; inf}\limits}

\begin{document}

\begin{center}
\Large \bf Boundedness of intrinsic square functions and their commutators
on generalized weighted Orlicz-Morrey spaces
\end{center}

\centerline{\large Vagif Guliyev$^{a,b,}$\footnote{
Corresponding author.
}, 
Mehriban Omarova$^{b,c}$, Yoshihiro Sawano$^{d}$}

\

{\bf Abstract.} We shall investigate the boundedness of the intrinsic square functions and their commutators
on generalized weighted Orlicz-Morrey spaces $M^{\Phi,\varphi}_{w}({\mathbb R}^n)$.
In all the cases, the conditions for the boundedness are given in terms of Zygmund-type integral inequalities
on weights $\varphi$ without assuming any monotonicity property of $\varphi(x,\cdot)$ with $x$ fixed.

\

\section{Introduction}

In the present paper, we are \tcr{concerned} with the intrinsic square functions,
which Wilson introduced initially \cite{Wilson1, Wilson2}.
For $0<\a\leq 1$, let $C_{\a}$ be the family of Lipschitz functions $\phi : \Rn\rightarrow \mathbb{R}$
of order $\alpha$ with the homogeneous norm $1$ such that the support of $\phi$
is contained in the closed ball $\{x:|x|\leq 1\}$, and that $\int_{\Rn}\phi(x) dx =0$.
For $(y,\ t)\in \mathbb{R}_{+}^{n+1}$ and $f\in L^{1,\rm loc}(\Rn)$, set
$$
A_{\a}f(t,\ y)\equiv\sup_{\phi\in C_{\a}}|f*\phi_{t}(y)|,
$$
where $\phi_{t}\equiv t^{-n}\phi\left(\frac{\cdot}{t}\right)$.
Let $\beta$ be an auxiliary parameter.
Then we define
the varying-aperture intrinsic square (intrinsic Lusin) function of $f$
by the formula;
$$
G_{\a,\beta}(f)(x)\equiv
\left(\iint_{\Gamma_{\beta}(x)}
(A_{\a}f(t,y))^{2}\frac{dydt}{t^{n+1}}\right)^{\frac{1}{2}},
$$
where
$\Gamma_{\beta}(x)\equiv\{(y,t) \in \mathbb{R}_{+}^{n+1}:|x-y|<\beta t\}$.
Write $G_{\a}(f)=G_{\a,1}(f)$ .

Everywhere in the sequel,
$B(x,r)$ stands for the ball in $\Rn$
of radius $r$ centered at $x$ and
we let $|B(x,r)|$ be the Lebesgue measure of the ball $B(x,r)$;
$|B(x,r)|=v_n r^n$, where $v_n$ is the volume of the unit ball in $\Rn$.
We recall generalized weighted Orlicz-Morrey spaces,
on which we work in the present paper.
\begin{defn}[Generalized weighted Orlicz-Morrey Space]
Let $\varphi$ be a positive measurable function
on $\Rn\times (0,\i)$,
let $w$ be non-negative measurable function on $\Rn$
and $\Phi$ any Young function.
Denote by $M^{\Phi,\varphi}_{w}({\mathbb R}^n)$
the generalized weighted Orlicz-Morrey space,
the space of all functions $f\in L^{\Phi,\rm loc}_{w}({\mathbb R}^n)$
such that
$$
\|f\|_{M^{\Phi,\varphi}_{w}} \equiv
\sup\limits_{x\in\Rn, r>0} \varphi(x,r)^{-1} \,
\Phi^{-1}\big(w(B(x,r))^{-1}\big) \, \|f\|_{L^{\Phi}_{w}(B(x,r))},
$$
where
$
\|f\|_{L^{\Phi}_{w}(B(x,r))}
\equiv
\inf\left\{\lambda>0\,:\,
\int_{B(x,r)}\Phi\left(\frac{|f(x)|}{\lambda}\right)w(x)\,dx \le 1
\right\}.
$
\end{defn}
According to this definition,
we recover the generalized weighted Morrey space
$M^{p,\varphi}_{w}({\mathbb R}^n)$
by the choice $\Phi(r)=r^{p},\,1\le p<\i$.
If $\Phi(r)=r^{p},\,1\le p<\i$ and $\varphi(x,r)=r^{-\frac{\lambda}{p}}$,
$0\le \lambda \le n$, then $M^{\Phi,\varphi}_{w}({\mathbb R}^n)$ coincides
with the weighted Morrey space $M^{p,\varphi}_{w}({\mathbb R}^n)$
and if $\varphi(x,r)=\Phi^{-1}(w(B(x,r)^{-1}))$,
then $M^{\Phi,\varphi}_{w}({\mathbb R}^n)$ coincides
with the weighted Orlicz space $L^{\Phi}_{w}({\mathbb R}^n)$.
When $w=1$, then $L^{\Phi}_{w}({\mathbb R}^n)$
is abbreviated to $L^{\Phi}({\mathbb R}^n)$.
The space $L^{\Phi}({\mathbb R}^n)$ is the classical Orlicz space.

Our first theorem of the present paper is the following one:
\begin{thm}\label{thm4.4.}
Let $\a\in(0,1]$ and $1<p_0\le p_1<\i$.
Let $\Phi$ be a Young function which is lower type $p_0$ and upper type $p_1$.
Namely,
\[
\Phi(st_0) \le Ct_0{}^{p_0}\Phi(s), \quad
\Phi(st_1) \le Ct_1{}^{p_1}\Phi(s)
\]
for all $s>0$ and $0<t_0 \le 1 \le t_1<\infty$.
Assume that $w \in A_{p_0}$ and
that the measurable functions $\varphi_1,\varphi_2:{\mathbb R}^n \times (0,\infty) \to (0,\infty)$ and $\Phi$ satisfy the condition;
\begin{equation}\label{eq3.6.VZPot}
\int_{r}^{\i}
\es_{t<s<\i}\frac{\varphi_1(x,s)}{\Phi^{-1}\big(w(B(x_0,s))^{-1}\big)}
\Phi^{-1}\big(w(B(x_0,t))^{-1}\big)\frac{dt}{t} \le C \,
\varphi_2(x,r),
\end{equation}
where $C$ does not depend on $x$ and $r$.
Then $G_{\a}$ is bounded
from $M^{\Phi,\varphi_1}_{w}(\Rn)$ to $M^{\Phi,\varphi_2}_{w}(\Rn)$.
\end{thm}

Theorem \ref{thm4.4.} extends
the result below due to Liang, Nakai, Yang and Zhou.
\begin{thm}{\rm \cite{LiNaYaZh}}\label{LiNaYaZhprop2.11-1}
Let $\a\in(0,1]$ and $1<p_0\le p_1<\i$.
Let $\Phi$ be a Young function which is lower type $p_0$ and upper type $p_1$.
Then $G_{\a}$ is bounded from $L^{\Phi}(\Rn)$ to itself.
\end{thm}

The function $G_{\a,\beta}(f)$ is independent of any particular kernel,
such as the Poisson kernel.
It dominates pointwise the classical square function
(Lusin area integral) and its real-variable generalizations.
Although the function $G_{\a,\beta}(f)$ depends on kernels
with uniform compact support, there is pointwise relation
between $G_{\a,\beta}(f)$ with different $\beta$:
$$
G_{\a,\beta}(f)(x)\leq\beta^{\frac{3n}{2}+\a}G_{\a}(f)(x)\ .
$$
See \cite{Wilson1} for details.

The intrinsic Littlewood-Paley $\mathrm{g}$-function is defined by
$$
g_{\a}f(x) \equiv \left(\int_{0}^{\i}(A_{\a}f(t,x))^{2}\frac{dt}{t}\right)^{\frac{1}{2}}.
$$
Also, the intrinsic $g_{\lambda,\a}^{*}$ function is defined
by
$$
g_{\lambda,\a}^{*}f(x)
\equiv
\left(\iint_{\mathbb{R}_{+}^{n+1}}
\left(\frac{t}{t+|x-y|}\right)^{n\lambda}(A_{\a}f(t,y))^{2}
\frac{dydt}{t^{n+1}}\right)^{\frac{1}{2}}.
$$
About this intrinsic Littlewood-Paley $\mathrm{g}$-function,
we shall prove the following boundedness property:
\begin{thm}\label{thm4.4.x}
Let $\a\in(0,1]$, $1<p_0\le p_1<\i$
and $\displaystyle \lambda\in\left(3+\frac{2\a}{n},\infty\right)$.
Let also
$\Phi$ be a Young function which is lower type $p_0$ and upper type $p_1$.
Assume that $w \in A_{p_0}$ and
that the functions
$\varphi_1,\varphi_2:{\mathbb R}^n \times (0,\infty) \to (0,\infty)$
and
$\Phi$ satisfy the condition \eqref{eq3.6.VZPot}.
Then $g_{\lambda,\a}^{*}$ is bounded
from $M^{\Phi,\varphi_1}_{w}(\Rn)$ to $M^{\Phi,\varphi_2}_{w}(\Rn)$.
\end{thm}

In \cite{Wilson1},
the author proved that the functions $G_{\a}f$ and $g_{\a}f$
are pointwise comparable.
Thus, as a consequence of Theorem \ref{thm4.4.},
we have the following result:
\begin{cor}\label{wucor1.5}
Let $\a\in(0,1]$ and $1<p_0\le p_1<\i$.
Let also $\Phi$ be a Young function
which is lower type $p_0$ and upper type $p_1$.
Assume in addition that $w \in A_{p_0}$ and
that the functions
$\varphi_1,\varphi_2:{\mathbb R}^n \times (0,\infty) \to (0,\infty)$
and $\Phi$ satisfy
the condition \eqref{eq3.6.VZPot}.
Then $g_{\a}$ is bounded
from $M^{\Phi,\varphi_1}_{w}(\Rn)$ to $M^{\Phi,\varphi_2}_{w}(\Rn)$.
\end{cor}

Let $b$ be a locally integrable function on $\mathbb{R}^{n}$.
Setting
$$
A_{\a,b}f(t,y)\equiv
\sup_{\phi\in C_{\a}}\left|\int_{\mathbb{R}^{n}}[b(y)-b(z)]
\phi_{t}(y-z)f(z)dz\right|,
$$
we can define the commutators
$[b,G_{\a}]$, $[b,g_{\a}]$ and $[b,g_{\lambda,\a}^{*}]$ by;
\begin{align*}
[b,G_{\a}]f(x)&\equiv
\left(\iint_{\Gamma(x)}(A_{\a,b}f(t,y))^{2}
\frac{dydt}{t^{n+1}}\right)^{\frac{1}{2}}\\
[b,g_{\a}]f(x)&\equiv
\left(\int_{0}^{\i}(A_{\a,b}f((t,x))^{2}\frac{dt}{t}\right)^{\frac{1}{2}}
\\
[b,g_{\lambda,\a}^{*}]f(x)
&\equiv
\left(\iint_{\mathbb{R}_{+}^{n+1}}
\left(\frac{t}{t+|x-y|}\right)^{\lambda n}
(A_{\a,b}f(t,y))^{2}\frac{dydt}{t^{n+1}}\right)^{\frac{1}{2}},
\end{align*}
respectively.
A function $f\in L^{1,\rm loc}(\mathbb{R}^{n})$ is said 
to be in ${\rm BMO}(\mathbb{R}^{n})$ \tcr{\cite{JohnNirenberg}} if 
$$
\Vert f\Vert_{*}
\equiv
\sup_{x\in \mathbb{R}^{n},r>0}
\frac{1}{|B(x,r)|}\int_{B(x,r)}|f(y)-f_{B(x,r)}|dy<\i,
$$
where $f_{B(x,r)}\equiv \displaystyle \frac{1}{|B(x,r)|}\int_{B(x,r)}f(y)dy$.

About the bounededness of $[b,G_{\a}]$ on Orlicz spaces, we shall invoke the following \tcr{result}:
\begin{thm}{\rm \cite{LiNaYaZh}}\label{LiNaYaZhprop2.17-1}
Let $\a\in(0,1]$, $1<p_0\le p_1<\i$ and $b\in {\rm BMO}(\Rn)$.
Let
$\Phi$ be a Young function which is lower type $p_0$ and upper type $p_1$
and $w \in A_{p_0}$.
Then $[b,G_{\a}]$ is bounded on $L^{\Phi}_{w}(\mathbb{R}^{n})$.
\end{thm}

About the commutator above,
we shall prove the following boundedness property
in the present paper:
\begin{thm}\label{3.4.XcomT}
Suppose that we are given parameters
$\a,p_0,p_1$
and functions
$b,w,\varphi,\varphi_2$ with the following properties:
\begin{enumerate}
\item
$\a\in(0,1], 1<p_0\le p_1<\i,$
\item
$b\in {\rm BMO}(\Rn)$
\item
$\Phi$ is a Young function which is lower type $p_0$ and upper type $p_1$.
\item
$w \in A_{p_0}$,
\item
$\varphi_1,\varphi_2$ and $\Phi$ satisfy the condition;
\begin{equation}\label{eq3.6.VZfrMaxcom}
\int_{r}^{\i}\Big(1+\ln \frac{t}{r}\Big)
\es_{t<s<\i}\frac{\varphi_1(x,s)\Phi^{-1}\big(w(B(x_0,t))^{-1}\big)}{\Phi^{-1}\big(w(B(x_0,s))^{-1}\big)}
\frac{dt}{t} \le C \, \varphi_2(x,r),
\end{equation}
where $C$ does not depend on $x$ and $r$.
\end{enumerate}
Then the operator $[b,G_{\a}]$ is bounded
from $M^{\Phi,\varphi_1}_{w}(\Rn)$ to $M^{\Phi,\varphi_2}_{w}(\Rn)$.
\end{thm}

In \cite{Wilson1}, the author proved
that the functions $G_{\a}f$ and $g_{\a}f$ are pointwise comparable.
From the definition of the commutators,
the same can be said for $[b,G_{\a}]$ and $[b,g_{\a}]$.
Thus, as a consequence of Theorem \ref{thm4.4.},
we have the following result:
\begin{cor}\label{wucor1.6}
Let $\a\in(0,1]$, $1<p_0\le p_1<\i$ and $b\in {\rm BMO}(\Rn)$.
Let $\Phi$ be a Young function which is lower type $p_0$ and upper type $p_1$.
Assume $w \in A_{p_0}$
and that the functions $\varphi_1,\varphi_2$ and $\Phi$
satisfy the condition \eqref{eq3.6.VZfrMaxcom}, then $[b,g_{\a}]$
is bounded from $M^{\Phi,\varphi_1}_{w}(\Rn)$ to $M^{\Phi,\varphi_2}_{w}(\Rn)$.
\end{cor}

\begin{rem}
By going through an argument similar to the above proofs and that
of Theorem \ref{thm4.4.x}, we can also show the boundedness
of $[b,g_{\lambda,\a}^{*}]$.
We omit the details.
\end{rem}

Here let us make a historical remark.
Wilson \cite{Wilson1} proved
that $G_{\a}$ is bounded on $L^{p}(\mathbb{R}^{n})$
for $1<p<\i$ and $0<\a\leq 1$.
After that,
Huang and Liu \cite{HuLiu} studied the boundedness
of intrinsic square functions on weighted Hardy spaces.
Moreover, they characterized the weighted Hardy spaces
by intrinsic square functions.
In \cite{Wang2} and \cite{WangLiu}, Wang and Liu obtained
some weak type estimates on weighted Hardy spaces.
In \cite{Wang1}, Wang considered intrinsic functions and
commutators generated by BMO functions on weighted Morrey spaces.
In \cite{WuX},
Wu proved the boundedness of intrinsic square functions and their commutators
 inspired by the ideas of Guliyev
\cite{GulDoc, GulBook, GulJIA, GULAKShIEOT2012}.
In \cite{LiNaYaZh}, Liang et al. studied the boundedness of these operators
on Musielak-Orlicz Morrey spaces.
Orlicz-Morrey spaces were initially introduced and studied
by Nakai in \cite{Nakai0}.
Also for the boundedness of the operators of harmonic analysis
on Orlicz-Morrey spaces, see also \cite{DerGulSam, HasJFSA, Nakai1, Nakai2, SawSugTan}.
Our definition of Orlicz-Morrey spaces (see \cite{DerGulSam}) is different from those
by Nakai \cite{Nakai0} and Sawano et al. \cite{SawSugTan} \tcr{used recently in  \cite{GalaRST2014}}.

Here and below, we use the following notations:
By $A \lesssim B$
we mean that $A \le C B$
with some positive constant $C$ independent of relavant quantities.
If $A \lesssim B$ and $B \lesssim A$, we write $A\approx B$ and
say that $A$ and $B$ are equivalent.

Finally, we descrive how we organize
the present paper.
In Section \ref{s2}
we recall some preliminary facts
such as Young functions and John-Nirenberg inequality.
Section \ref{s3} is devoted to the proof
of Theorems \ref{thm4.4.} and \ref{thm4.4.x}.
We prove Theorem \ref{3.4.XcomT} in Section \ref{s4}.

\section{Preliminaries}
\label{s2}

As is well known,
classical Morrey spaces stemmed from Morrey's observation
for the local behavior of solutions
to second order elliptic partial differential equations \cite{Morrey}.
We recall its definition:
\begin{equation*}
M_{p,\lambda}(\Rn) = \left\{ f \in L^{p,\rm loc}(\Rn) : \left\| f\right\|_{M_{p,\lambda}}: = \sup_{x \in \Rn, \; r>0 } r^{-\frac{\lambda}{p}} \|f\|_{L^{p}(B(x,r))} < \i \right\},
\end{equation*}
where $0 \le \lambda \le n,$ $1\le p < \i$.
The scale $M_{p,\l}(\Rn)$ covers the $L^p(\Rn)$
in the sense that $M_{p,0}(\Rn)=L^p(\Rn)$.

We are thus oriented to a generalization
of the parameters $p$ and $\lambda$.

\subsection{Young functions and Orlicz spaces}

We next recall the definition of Young functions.

\begin{defn}\label{def2}
A function $\Phi : [0,+\i) \rightarrow [0,\i]$ is called a Young function,
if $\Phi$ is convex, left-continuous,
$\lim\limits_{r\rightarrow +0} \Phi(r) = \Phi(0) = 0$
and
$\lim\limits_{r\rightarrow +\i} \Phi(r) = \i$.
\end{defn}
The convexity and the condition $\Phi(0) = 0$
force any Young function to be increasing.
In particular,
if there exists $s \in (0,+\i)$ such that $\Phi(s) = +\i$,
then it follows that $\Phi(r) = +\i$ for $r \geq s$.

Let $\mathcal{Y}$ be the set of all Young functions $\Phi$ such that
\begin{equation}\label{2.1}
0<\Phi(r)<+\i\qquad \text{for} \qquad 0<r<+\i
\end{equation}
If $\Phi \in \mathcal{Y}$,
then $\Phi$ is absolutely continuous on every closed interval in $[0,+\i)$
and bijective from $[0,+\i)$ to itself.

Orlicz spaces,
introduced in \cite{Orlicz1, Orlicz2},
also
generalize Lebesgue spaces. 
They are useful tools in harmonic analysis and
these spaces are applied to many other problems in harmonic analysis.
For example,
the Hardy-Littlewood maximal operator is bounded on $L^p({\mathbb R}^n)$
 for $1 < p < \i$,
but not on $L^1({\mathbb R}^n)$.
Using Orlicz spaces, we can investigate the boundedness of the maximal operator near $p = 1$ more precisely.

In the present paper we are concerned with the weighted setting.
\begin{defn}[Weighted Orlicz Space]\label{ttss}
For a Young function $\Phi$
and a non-negative measurable function $w$ on $\Rn$,
the set
$$L^{\Phi}_{w}(\Rn)
\equiv
\left\{f\in L^{\Phi,\rm loc}_{w}(\Rn):
\int_{\Rn}\Phi(k|f(x)|) w(x)dx<+\i \text{ for some $k>0$ }\right\}
$$
is called the weighted Orlicz space.
The local weighted Orlicz space $L^{\Phi,\rm loc}_{w}(\Rn)$
is defined as the set of all functions $f$
such that $f\chi_{_B}\in L^{\Phi}_{w}(\Rn)$ for all balls $B \subset \Rn$
and this space is endowed with the natural topology.
\end{defn}
Note that $L^{\Phi}_{w}(\Rn)$ is a Banach space
with respect to the norm
$$
\|f\|_{L^{\Phi}_{w}}
\equiv
\inf\left\{\lambda>0:
\int_{\Rn}\Phi\Big(\frac{|f(x)|}{\lambda}\Big)w(x)dx\leq 1\right\}.
$$
See \cite[Section 3, Theorem 10]{RaoRen}  for example.
In particular, we have
$$
\int_{\Rn}\Phi\Big(\frac{|f(x)|}{\|f\|_{L^{\Phi}_{w}}}\Big)w(x)dx\leq 1.
$$

If $\Phi(r)=r^{p},\, 1\le p<\i$,
then $L^{\Phi}_{w}=L^{p}_{w}(\Rn)$
with norm coincidence.
If $\Phi(r)=0,\,(0\le r\le 1)$
and
$\Phi(r)=\i,\,(r> 1)$, then $L^{\Phi}_{w}=L^{\i}_{w}(\Rn)$.

For a Young function $\Phi$ and $0 \leq s \leq +\i$, let
$$
\Phi^{-1}(s)
\equiv
\inf\{r\geq 0: \Phi(r)>s\}\qquad (\inf\emptyset=+\i).$$
If $\Phi \in \mathcal{Y}$,
then $\Phi^{-1}$ is the usual inverse function of $\Phi$.
We also note that
\begin{equation}\label{younginverse}
\Phi(\Phi^{-1}(r))
\leq r \leq
\Phi^{-1}(\Phi(r)) \quad \text{ for } 0\leq r<+\i.
\end{equation}
A Young function $\Phi$ is said to satisfy the $\Delta_2$-condition, denoted by $\Phi \in \Delta_2$, if
$$
\Phi(2r)\le k\Phi(r) \text{  for } r>0
$$
for some $k>1$.
If $\Phi \in \Delta_2$, then $\Phi \in \mathcal{Y}$.
A Young function $\Phi$ is said to satisfy the $\nabla_2$-condition,
denoted also by $\Phi \in \nabla_2$, if
$$\Phi(r)\leq \frac{1}{2k}\Phi(kr),\qquad r\geq 0,$$
for some $k>1$.
The function $\Phi(r) = r$ satisfies the $\Delta_2$-condition
and it fails the $\nabla_2$-condition.
If $1 < p < \i$,
then $\Phi(r) = r^p$ satisfies both the conditions.
The function $\Phi(r) = e^r - r - 1$ satisfies the
$\nabla_2$-condition but it fails the $\Delta_2$-condition.

\begin{defn}
A Young function $\Phi$ is said to be
of upper type $p$ (resp. lower type $p$) for some $p\in[0,\i)$,
if there exists a positive constant $C$ such that,
for all $t\in[1,\i)$ $($resp. $t\in[0,1]$ $)$ and $s\in[0,\i)$,
$$
\Phi(st)\le Ct^p\Phi(s).
$$
\end{defn}

\begin{rem}\label{remlowup}
If $\Phi$ is lower type $p_0$ and upper type $p_1$ with $1<p_0\le p_1<\i$,
then $\Phi\in \Delta_2\cap\nabla_2$.
Conversely if $\Phi\in \Delta_2\cap\nabla_2$,
then $\Phi$ is lower type $p_0$ and upper type $p_1$ with $1<p_0\le p_1<\i$;
see \cite{KokKrbec} for example.
\end{rem}

About the norm $\|f\|_{M^{\Phi,\varphi}_{w}}$,
we have the following equivalent expression:
If $\Phi$ satisfies the $\Delta_2$-condition,
then the norm $\|f\|_{M^{\Phi,\varphi}_{w}}$ is equivalent
to the norm
\begin{align*}
\|f\|_{\overline{M}^{\Phi,\varphi}(w)}
\equiv \inf\Big\{\lambda>0 : &
\sup\limits_{x\in\Rn, r>0} \varphi(x,r)^{-1} \, \Phi^{-1}\big(w(B(x,r))^{-1}\big)
\\
& \times \int\limits_{B(x,r)}\Phi\Big(\frac{|f(x)|}{\lambda}\Big)w(x)dx \le 1 \Big\}.
\end{align*}
See \cite[p. 416]{MizNakaiOhnoShim}.
The latter was used in \cite{MizNakaiOhnoShim, Nakai1, Nakai2, SawSugTan},
see also references therein.
%
%
%
%
For $\Phi$ and $\widetilde{\Phi}$,
we have the following estimate,
whose proof is similar to \cite[Lemmas 4.2]{Ky1}.
So, we omit the details.

\begin{lem} \label{Kylowupp}
Let $0<p_0\le p_1<\i$ and let $\widetilde{C}$ be a positive constant.
Suppose that we are given
a non-negative measurable function $w$ on $\Rn$
and a Young function $\Phi$ which is lower type $p_0$ and upper type $p_1$.
Then there exists a positive constant $C$ such that
for any ball $B$ of $\Rn$ and $\mu\in(0,\i)$
$$\int_{B}\Phi\left(\frac{|f(x)|}{\mu}\right) w(x) dx\le \widetilde{C}$$
implies that $\|f\|_{L^{\Phi}_{w}(B)}\le C\mu$.
\end{lem}

For a Young function $\Phi$,
the complementary function $\widetilde{\Phi}(r)$ is defined by
\begin{equation}\label{2.2}
\widetilde{\Phi}(r)
\equiv
\left\{
\begin{array}{ccc}
\sup\{rs-\Phi(s): s\in [0,\i)\}
& \mbox{ if } & r\in [0,\i), \\
+\i&\mbox{ if }& r=+\i.
\end{array}
\right.
\end{equation}
The complementary function $\widetilde{\Phi}$
is also a Young function and
it satisfies $\widetilde{\widetilde{\Phi}}=\Phi$.
Here we recall three examples.
\begin{ex}
\
\begin{enumerate}
\item
If $\Phi(r)=r$, then $\widetilde{\Phi}(r)=0$ for $0\leq r \leq 1$
and $\widetilde{\Phi}(r)=+\i$ for $r>1$.
\item
If $1 < p < \i$, $1/p+1/p^\prime= 1$ and $\Phi(r) =r^p/p$,
then $\widetilde{\Phi}(r) = r^{p^\prime}/p^\prime$.
\item
If $\Phi(r) = e^r-r-1$, then a calculation shows
$\widetilde{\Phi}(r) = (1+r) \log(1+r)-r.$
\end{enumerate}
\end{ex}
Note that $\Phi \in \nabla_2$ if and only if $\widetilde{\Phi} \in \Delta_2$.
It is also known that
\begin{equation}\label{2.3}
r\leq \Phi^{-1}(r)\widetilde{\Phi}^{-1}(r)\leq 2r
\qquad \text{for } r\geq 0.
\end{equation}
Note that Young functions satisfy the properties;
\begin{equation} \label{sam1}
\Phi(\a t)\leq \a \Phi(t)
\end{equation}
for all $0\le\a\le1$ and $0 \le t < \i$, and
\begin{equation} \label{sam2}
\Phi(\beta t)\geq \beta \Phi(t)
\end{equation}
for all $\beta>1$ and $0 \le t < \i$.

The following analogue of the H\"older inequality is known,
see \cite{Weiss}.
\begin{thm}{\rm \cite{Weiss}} \label{HolderOr}
For a non-negative measurable function $w$ on $\Rn$, a Young function $\Phi$
and its complementary function $\widetilde{\Phi}$,
the following inequality is valid
for all measurable functions $f$ and $g$:
$
\|fg\|_{L^{1}(\Rn)} \leq 2
\|f\|_{L^{\Phi}_{w}}
\|w^{-1}g\|_{L^{\widetilde{\Phi}}_{w}}.
$
\end{thm}

An analogy of Theorem \ref{HolderOr} for weak type spaces is available.
If we define
\[
\|f\|_{WL^{\Phi}_{w}}
\equiv
\sup_{\lambda>0}\lambda\|\chi_{\{|f|>\lambda\}}\|_{L^{\Phi}_{w}},
\]
we can prove the following by a direct calculation:
\tcr{
\begin{cor}\label{lem4.0}
Let $\Phi$ be a Young function and
let $B$ be a measurable set in $\Rn$. Then
$
\|\chi_{_B}\|_{WL^{\Phi}_{w}}
= \|\chi_{_B}\|_{L^{\Phi}_{w}}
= \frac{1}{\Phi^{-1}\left(w(B)^{-1}\right)}.
$
\end{cor}
}

In the next sections where we prove our main estimates,
we need the following lemma, which follows from Theorem \ref{HolderOr}. 
\begin{cor}
For a non-negative measurable function $w$ on $\Rn$,
a Young function $\Phi$ and a ball $B=B(x,r)$,
the following inequality is valid:
$$
\|f\|_{L^{1}(B)} \leq 2  \Big\|\frac{1}{w}\Big\|_{L^{\widetilde{\Phi}}_{w}(B)} \, \|f\|_{L^{\Phi}_{w}(B)}.
$$
\end{cor}

\begin{lem}\label{lemHold}
Let $\a\in(0,1]$ and $1<p_0\le p_1<\i$.
Let also $\Phi$ be a Young function
which is lower type $p_0$ and upper type $p_1$.
Assume in addition $w \in A_{p_0}$.
For a ball $B=B(x,r)$,
the following inequality is valid:
$$
\|f\|_{L^{1}(B)} \lesssim |B|
\Phi^{-1}\left(w(B)^{-1}\right) \|f\|_{L^{\Phi}_{w}(B)}.
$$
\end{lem}

\begin{proof}
We know that $M$ is bounded on $L^{\Phi}_w(B)$; see \cite{KermTorch}. Thus,
\[
\frac{\|f\|_{L^1(B)}}{|B|}
\|\chi_B\|_{L^\Phi_w(B)}
\le
\|Mf\|_{L^\Phi_w(B)}
\lesssim
\|f\|_{L^{\Phi}_{w}(B)}.
\]
\tcr{So, Lemma \ref{lemHold} is proved}.
\end{proof}

\subsection{Weighted Hardy operator}

We will use the following statement on the boundedness
of the weighted Hardy operator
$$
H^{\ast}_{w} g(t):=\int_t^{\i} g(s) w(s) ds,~ \ \ 0<t<\i,
$$
where $w$ is a weight.

The following theorem was proved in \cite{GulJMS2013}.
In \eqref{vav01} and \eqref{vav02} below,
it will be understood that $\frac{1}{\i}=0$ and $0 \cdot \i=0$.

\begin{thm}\label{thm3.2.}
Let $v_1$, $v_2$ and $w$ be weights on $(0,\i)$.
Assume that $v_1$ is bounded
outside a neighborhood of the origin.
Then the inequality
\begin{equation} \label{vav01}
\sup _{t>0} v_2(t) H^{\ast}_{w} g(t) \leq C \sup _{t>0} v_1(t) g(t)
\end{equation}
holds for some $C>0$
for all non-negative and non-decreasing $g$ on $(0,\i)$
if and only if
\begin{equation} \label{vav02}
B:= \sup _{t>0} v_2(t)
\int_t^{\i} \frac{w(s) ds}{\sup _{s<\tau<\i} v_1(\tau)}<\i.
\end{equation}
Moreover, the value $C=B$ is the best constant for \eqref{vav01}.
\end{thm}

\subsection{John-Nirenberg inequality}

When we deal with commutators generated by BMO functions,
we need the following fundamental estimates.
\begin{lem}{\rm(The John--Nirenberg inequality \cite{JohnNirenberg})}\label{rem2.4.}
Let $b \in {\rm BMO}(\Rn)$.
\begin{enumerate}
\item[$(1)$]
There exist constants $C_1$, $C_2>0$
independent of $b$, such that
$$
\left| \left\{ x \in B \, : \, |b(x)-b_{B}|>\beta \right\}\right|
\le C_1 |B| e^{-C_2 \beta/\| b \|_{\ast}}, ~~~ \forall B \subset \Rn
$$
for all $\beta>0$.
\item[$(2)$]
The following norm equivalence holds:
\begin{equation} \label{lem2.4.xx}
\|b\|_\ast \thickapprox \sup_{x\in\Rn, r>0}\left(\frac{1}{|B(x,r)|}
\int_{B(x,r)}|b(y)-b_{B(x,r)}|^p dy\right)^{\frac{1}{p}}
\end{equation}
for $1<p<\i$.
\item[$(3)$]
There exists a constant $C>0$ such that
\begin{equation} \label{propBMO}
\left|b_{B(x,r)}-b_{B(x,t)}\right| \le C
\|b\|_\ast \ln \frac{t}{r} \;\;\; \mbox{for} \;\;\; 0<2r<t,
\end{equation}
where $C$ is independent of $b$, $x$, $r$ and $t$.
\end{enumerate}
\end{lem}
\section{Intrinsic square functions in $M^{\Phi,\varphi}_{w}(\Rn)$}
\label{s3}

The following lemma generalizes Guliyev's lemma
\cite{GulDoc, GulBook, GulJIA} for Orlicz spaces:
\begin{lem}\label{lem3.3.}
Let $\a\in(0,1]$ and $1<p_0\le p_1<\i$.
Let $\Phi$ be a Young function
which is lower type $p_0$ and upper type $p_1$.
Assume that the weight belongs to the class $w \in A_{p_0}$.
Then for the operator $G_{\a}$ the following inequality is valid:
\begin{equation}\label{CZdgs}
\|G_{\a} f\|_{L^{\Phi}_{w}(B)} \lesssim
\int_{2r}^{\i} \|f\|_{L^{\Phi}(B(x_0,t))}
\frac{\Phi^{-1}\big(w(B(x_0,t))^{-1}\big)}{\Phi^{-1}\big(w(B(x_0,r))^{-1}\big)}\frac{dt}{t}
\end{equation}
for all
$f\in L^{\Phi,\rm loc}_{w}(\Rn)$, $B=B(x_0,r)$, $x_0\in \Rn$ and $r>0$.
\end{lem}

\begin{proof}
With the notation $2B=B(x_0,2r)$, we decompose $f$ as
$$
f=f_1+f_2, \ \quad
f_1(y)\equiv f(y)\chi _{2B}(y),\quad
f_2(y)\equiv f(y)\chi_{\dual {(2B)}}(y).
$$
We have
$$
\|G_{\a}f\|_{L^{\Phi}_{w}(B)}
\le
\|G_{\a}f_1\|_{L^{\Phi}_{w}(B)}
+
\|G_{\a}f_2\|_{L^{\Phi}_{w}(B)}
$$
by the triangle inequality.
Since $f_1\in L^{\Phi}_{w}(\Rn)$,
it follows from Theorem \ref{LiNaYaZhprop2.11-1} that
\begin{equation}\label{eq:140507-7}
\|G_{\a}f_1\|_{L^{\Phi}_{w}(B)}
\leq
\|G_{\a}f_1\|_{L^{\Phi}_{w}(\Rn)}
\lesssim
\|f_1\|_{L^{\Phi}_{w}(\Rn)}
=
\|f\|_{L^{\Phi}_{w}(2B)}.
\end{equation}
So, we can control $f_1$.

Now let us estimate $\|G_{\a}f_2\|_{L^{\Phi}_{w}(B)}$.
Let $x \in B=B(x_0,r)$ and write out $G_{\a}f_2(x)$ in full:
\begin{equation}\label{eq:140507-3}
G_{\a}(f)(x)\equiv
\left(\iint_{\Gamma(x)}
\left(\sup_{\phi\in C_{\a}}|f_2*\phi_{t}(y)|\right)^{2}\frac{dydt}{t^{n+1}}\right)^{\frac{1}{2}}.
\end{equation}
Let $(y,t) \in \Gamma(x)$.
We next write the convolution
$f_{2}*\displaystyle \phi_{t}(y)$ out in full:
\begin{equation}\label{eq:140507-2}
|f_{2}*\displaystyle \phi_{t}(y)|
=
\left|t^{-n}\int_{|y-z|\leq t}\phi\left(\frac{y-z}{t}\right)f_{2}(z)dz\right|
\lesssim
\frac{1}{t^n}\int_{|y-z|\leq t}|f_{2}(z) |dz.
\end{equation}
Recall that the suppost of $f$ is contained in $\dual {(2B)}$.
Keeping this in mind,
let $z \in B(y,t) \cap \dual{(2B)}$.
Since $(y,t)\in\Gamma(x)$,
we have
\begin{equation}\label{eq:140507-1}
|z-x|\leq|z-y|+|y-x|\leq 2t.
\end{equation}
Another geometric observation shows
$$r=2r-r\leq|z-x_{0}|-|x_{0}-x|\leq|x-z|.$$
Thus,
we obtain
\begin{equation}\label{eq:140507-5}
2t \ge r
\end{equation}
from (\ref{eq:140507-1}).
So,
putting together (\ref{eq:140507-3})--(\ref{eq:140507-5}),
we obtain
\begin{eqnarray*}
G_{\a}f_{2}(x) &\lesssim&
\left(\int \int_{\Gamma(x)}
\left|t^{-n}\int_{|y-z|\leq t}|f_{2}(z)|dz\right|^{2}\frac{dydt}{t^{n+1}}
\right)^{\frac{1}{2}}
\\
{}&\leq&\left(\int_{t>r/2}\int_{|x-y|<t}
\left(\int_{|z-x|\leq 2t}|f(z)|dz
\right)^{2}\frac{dydt}{t^{3n+1}}\right)^{\frac{1}{2}}
\\
{}&\lesssim&
\left(\int_{t>r/2}\left(\int_{|z-x|\leq 2t}|f(z)|dz\right)^{2}
\frac{dt}{t^{2n+1}}\right)^{\frac{1}{2}}.
\end{eqnarray*}
We make another geometric observation:
\begin{equation}\label{eq:140507-4}
|z-x|\displaystyle \geq|z-x_{0}|-|x_{0}-x|\geq\frac{1}{2}|z-x_{0}|.
\end{equation}
By Minkowski's inequality, we obtain
\[
G_{\a}f_{2}(x)
\lesssim
\int_{\mathbb{R}^{n}}\left(\int_{t>\frac{|z-x|}{2}}\frac{dt}{t^{2n+1}}
\right)^{\frac{1}{2}}|f(z)|dz.
\]
Thanks to (\ref{eq:140507-4}),
we have
\begin{eqnarray*}
G_{\a}f_{2}(x)
{}&\lesssim&\int_{|z-x_{0}|>2r}\frac{|f(z)|}{|z-x|^{n}}dz\\
&\lesssim&
\int_{|z-x_{0}|>2r}\frac{|f(z)|}{|z-x_{0}|^{n}}dz \nonumber
\\
{}&=&
\int_{|z-x_{0}|>2r}|f(z)|
\left(\int_{|z-x_{0}|}^{+\i}\frac{dt}{t^{n+1}}\right)dz\\
&=&\int_{2r}^{\i}\left(\int_{B(x_{0},t)}|f(z)|dz\right)
\frac{dt}{t^{n+1}}.
\end{eqnarray*}
If we invoke Lemma \ref{lemHold},
then we obtain
\begin{eqnarray}\label{estGf2}
G_{\a}f_{2}(x) &\lesssim&
\int_{2r}^{\i} \|f\|_{L^{\Phi}_{w}(B(x_0,t))}
\Phi^{-1}\big(w(B(x_0,t))^{-1}\big) \frac{dt}{t}.
\end{eqnarray}
Moreover,
\begin{equation} \label{ves2}
\|G_{\a}f_2\|_{L^{\Phi}_{w}(B)}\lesssim
\int_{2r}^{\i}\|f\|_{L^{\Phi}_{w}(B(x_0,t))}
\frac{\Phi^{-1}\big(w(B(x_0,t))^{-1}\big)}
{\Phi^{-1}\big(w(B(x_0,r))^{-1}\big)}
\frac{dt}{t}.
\end{equation}
Thus, it follows from (\ref{eq:140507-7}) and (\ref{estGf2}) that
\begin{equation}\label{eq:140507-8}
\|G_{\a}f\|_{L^{\Phi}_{w}(B)}
\lesssim
\|f\|_{L^{\Phi}_{w}(2B)}+
\int_{2r}^{\i}\|f\|_{L^{\Phi}_w (B(x_0,t))}
\frac{\Phi^{-1}\big(w(B(x_0,t))^{-1}\big)}
{\Phi^{-1}\big(w(B(x_0,r))^{-1}\big)}
\frac{dt}{t}.
\end{equation}
On the other hand, by \eqref{2.3} we get
\begin{align*} 
\Phi^{-1}\big(w(B(x_0,r))^{-1}\big) &\thickapprox
\Phi^{-1}\big(w(B(x_0,r))^{-1}\big) r^n \int_{2r}^{\i}\frac{dt}{t^{n+1}} \notag
\\
& \lesssim
\int_{2r}^{\i} \Phi^{-1}\big(w(B(x_0,t))^{-1}\big) \frac{dt}{t}
\end{align*}
and hence
\begin{equation} \label{ves2fd}
\|f\|_{L^{\Phi}_{w}(2B)}\lesssim
\int_{2r}^{\i} \|f\|_{L^{\Phi}_{w}(B(x_0,t))}
\frac{\Phi^{-1}\big(w(B(x_0,t))^{-1}\big)}
{\Phi^{-1}\big(w(B(x_0,r))^{-1}\big)}
 \frac{dt}{t}.
\end{equation}
Thus, it follows from (\ref{eq:140507-8}) and (\ref{ves2fd})
that
\begin{equation*}
\|G_{\a}f\|_{L^{\Phi}_{w}(B)}
\lesssim
\int_{2r}^{\i}
\|f\|_{L^{\Phi}_{w} (B(x_0,t))}
\frac{\Phi^{-1}\big(w(B(x_0,t))^{-1}\big)}
{\Phi^{-1}\big(w(B(x_0,r))^{-1}\big)} \frac{dt}{t}.
\end{equation*}
So, we are done.
\end{proof}

With this preparation,
we can prove Theorem \ref{thm4.4.}

\begin{proof}
Fix $x \in {\mathbb R}^n$.
Write
\begin{gather*}
v_1(r)\equiv \varphi_1(x,r)^{-1}, \quad
v_2(r)\equiv\frac{1}{\varphi_2(x,r)\Phi^{-1}(w(B(x_0,r))^{-1})}, \quad\\
g(r)\equiv\|f\|_{L^{\Phi}_{w}(B(x_0,r))}, \quad
\omega(r)\equiv\frac{\Phi^{-1}(w(B(x_0,r))^{-1})}{r}.
\end{gather*}
We omit a routine procude of truncation
to justify the application of Theorem \ref{thm3.2.}.
By Lemma \ref{lem3.3.} and Theorem \ref{thm3.2.},
we have
\begin{align*}
\lefteqn{
\|G_{\a} f\|_{M^{\Phi,\varphi_2}_{w}(\Rn)}
}\\
& \lesssim \sup_{x\in\Rn,\,r>0}
\frac{1}{\varphi_2(x,r)} \int_{r}^{\i}\|f\|_{L^{\Phi}_{w}(B(x_0,t))}
\Phi^{-1}\big(w(B(x_0,t))^{-1}\big) \frac{dt}{t}
\\
& \lesssim \sup_{x\in\Rn,\,r>0}
\frac{1}{\varphi_1(x,r)} \, \Phi^{-1}\big(w(B(x_0,r))^{-1}\big)
\, \|f\|_{L^{\Phi}_{w}(B(x_0,r))}
\\
& = \|f\|_{M^{\Phi,\varphi_1}}.
\end{align*}
So we are done.
\end{proof}

The following lemma is an easy consequence of
the monotonicity of the norm $\|\cdot\|_{L^{\Phi}_{w}}$
and Wilson's estimate;
$$
G_{\a,\beta}(f)(x)\leq\beta^{\frac{3n}{2}+\a}G_{\a}(f)(x)
\quad (x \in \R^n),
$$
which was proved in \cite{Wilson1}.
\begin{lem}\label{wulem2.3}
For $j\in \mathrm{Z}^{+}$, denote
$$
G_{\a,2^{j}}(f)(x)
\equiv
\left(\int_{0}^{\i}\int_{|x-y|\leq 2^{j}t}
(A_{\a}f(t,y))^{2}\frac{dydt}{t^{n+1}}\right)^{\frac{1}{2}}
$$
Let $\Phi$ be a Young function and $0<\a\leq 1$.
Then we have
$$
\Vert G_{\a,2^{j}}(f)\Vert_{L^{\Phi}_{w}}
\lesssim 2^{j(\frac{3n}{2}+\a)}
\Vert G_{\a}(f)\Vert_{L^{\Phi}_{w}}
$$
for all $f \in L^{\Phi}_w({\mathbb R}^n)$.
\end{lem}

Now we can prove Thoerem \ref{thm4.4.x}.
We write $g_{\lambda,\a}^{*}(f)(x)$ out in full:
\begin{eqnarray*}
[g_{\lambda,\a}^{*}(f)(x)]^{2} =
\displaystyle \iint_{\Gamma(x)}+
\iint_{\,^{^{\complement}}\!\Gamma(x)}
\left(\frac{t}{t+|x-y|}\right)^{n\lambda}
(A_{\a}f(t,y))^{2}\frac{dydt}{t^{n+1}}
:=I+II.
\end{eqnarray*}
As for $I$,
a crude estimate suffices;
\begin{equation}\label{eq:140420-1}
I\leq \displaystyle \iint_{\Gamma(x)}
(A_{\a}f(t,y))^{2}\frac{dydt}{t^{n+1}}\leq(G_{\a}f(x))^{2}.
\end{equation}
Thus, the heart of the matters is to control $II$.
We decompose the ambient space ${\mathbb R}^n$:
\begin{eqnarray}
II &\leq&
\sum_{j=1}^{\i}\int_{0}^{\i}\int_{2^{j-1}t\leq|x-y|\leq 2^{j}t}
\left(\frac{t}{t+|x-y|}\right)^{n\lambda}(A_{\a}f(t,y))^{2}
\frac{dydt}{t^{n+1}}
\nonumber\\
{}&\lesssim &
\sum_{j=1}^{\i}\int_{0}^{\i}\int_{2^{j-1}t\leq|x-y|\leq 2^{j}t}2^{-jn\lambda}
(A_{\a}f(t,y))^{2}\frac{dydt}{t^{n+1}}
\nonumber\\
{}&\lesssim &
\sum_{j=1}^{\i}\iint_{\Gamma_{2^j}(x)}
\frac{(A_{\a}f(t,y))^{2}}{2^{jn\lambda}}\frac{dydt}{t^{n+1}}
:=\sum_{j=1}^{\i}\frac{(G_{\a,2^{j}}(f)(x))^{2}}{2^{jn\lambda}}.
\label{eq:140420-2}
\end{eqnarray}
Thus, putting together (\ref{eq:140420-1}) and (\ref{eq:140420-2}),
we obtain
\begin{equation}\label{wueq7}
\Vert g_{\lambda,\a}^{*}(f) \Vert_{M^{\Phi,\varphi_2}_{w}}
\lesssim
\Vert G_{\a}f \displaystyle \Vert_{M^{\Phi,\varphi_2}_{w}}
+\sum_{j=1}^{\i} 2^{-\frac{jn\lambda}{2}}
\Vert G_{\a,2^{j}}(f) \Vert_{M^{\Phi,\varphi_2}_{w}}.
\end{equation}
By Theorem \ref{thm4.4.}, we have
\begin{equation}\label{wueq8}
\Vert G_{\a}f\Vert_{M^{\Phi,\varphi_2}_{w}(\Rn)}\lesssim\Vert f\Vert_{M^{\Phi,\varphi_1}_{w}(\Rn)}.
\end{equation}
In the sequel,
we will estimate $\Vert G_{\a,2^{j}}(f) \Vert_{M^{\Phi,\varphi_2}_{w}}$.
We divide $\Vert G_{\a,2^{j}}(f)\Vert_{L^{\Phi}_{w}(B)}$ into two parts:
\begin{equation}\label{wueq9}
\Vert G_{\a,2^{j}}(f)\Vert_{L^{\Phi}_{w}(B)}\leq
\Vert G_{\a,2^{j}}(f_{1})\Vert_{L^{\Phi}_{w}(B)}
+
\Vert G_{\a,2^{j}}(f_{2})\Vert_{L^{\Phi}_{w}(B)},
\end{equation}
where $f_{1}(y)\equiv f(y)\chi_{2B}(y)$ and $f_{2}(y) \equiv f(y)-f_{1}(y)$.
For $\Vert G_{\a,2^{j}}(f_{1})\Vert_{L^{\Phi}_{w}(B)}$,
by Lemma \ref{wulem2.3} and \eqref{ves2fd},
we have (see also, \cite[p. 47, (5.4)]{GulEMJ2012})
\begin{align}
\nonumber
\Vert G_{\a,2^{j}}(f_{1})\Vert_{L^{\Phi}_{w}(B)}
&\lesssim
2^{j(\frac{3n}{2}+\a)}
\Vert G_{\a}(f_{1})\Vert_{L^{\Phi}_{w}(\mathbb{R}^{n})}\\
&\lesssim
2^{j(\frac{3n}{2}+\a)}\Vert f\Vert_{L^{\Phi}_{w}(2B)}\\
\label{wueq10}
&\lesssim
2^{j(\frac{3n}{2}+\a)}
\int_{2r}^{\i} \|f\|_{L^{\Phi}_{w}(B(x_0,t))}
\frac{\Phi^{-1}\big(w(B(x_0,t))^{-1}\big)}
{\Phi^{-1}\big(w(B(x_0,r))^{-1}\big)}
 \frac{dt}{t}.
\end{align}
For $\Vert G_{\a,2^{j}}(f_{2})\Vert_{L^{\Phi}_{w}(B)}$,
we first write the quantity out in full:
\begin{eqnarray*}
G_{\a,2^{j}}(f_{2})(x)
&=&
\left(\displaystyle \iint_{\Gamma_{2^j}(x)}
(A_{\a}f(t,y))^{2}\frac{dydt}{t^{n+1}}\right)^{\frac{1}{2}}
\\
{}&=& \left(\iint_{\Gamma_{2^j}(x)}
\left(\sup_{\phi\in C_{\a}}|f*\phi_{t}(y)|\right)^{2}
\frac{dydt}{t^{n+1}}\right)^{\frac{1}{2}}.
\end{eqnarray*}
A geometric observation shows that
\begin{eqnarray*}
G_{\a,2^{j}}(f_{2})(x)
{}&\lesssim& \left(\iint_{\Gamma_{2^j}(x)}
\left(\int_{|z-y|\leq t}|f_{2}(z)|dz\right)^{2}\frac{dydt}{t^{3n+1}}
\right)^{\frac{1}{2}}.
\end{eqnarray*}
Since $|z-x|\leq|z-y|+|y-x|\leq 2^{j+1}t$, we get
\begin{eqnarray*}
G_{\a,2^{j}}(f_{2})(x) &\lesssim&
\left(\displaystyle \iint_{\Gamma_{2^j}(x)}
\left(\int_{|z-x|\leq 2^{j+1}t}|f_{2}(z)|dz\right)^{2}
\frac{dydt}{t^{3n+1}}\right)^{\frac{1}{2}}
\\
{}&\lesssim& \left(\int_{0}^{\i}
\left(\int_{|z-x|\leq 2^{j+1}t}|f_{2}(z)|dz\right)^{2}
\frac{2^{jn}dt}{t^{2n+1}}\right)^{\frac{1}{2}}
\\
{}&\lesssim& 2^{\frac{jn}{2}}\int_{\mathbb{R}^{n}}
\left(\int_{\frac{|z-x|}{2^{j+1}}}^{\infty}
\frac{|f_{2}(z)|^{2}}{t^{2n+1}}dt\right)^{\frac{1}{2}}dz
\lesssim 2^{\frac{3jn}{2}}\int_{\dual{B(x_0,2r)}}\frac{|f(z)|dz}{|z-x|^{n}}.
\end{eqnarray*}
A geometric observation shows
$$
|z-x|
\geq|z-x_{0}|-|x_{0}-x|\geq|z-x_{0}|-\frac{1}{2}|z-x_{0}|
=\frac{1}{2}|z-x_{0}|.
$$
Thus,
we have
\[
G_{\a,2^{j}}(f_{2})(x)\lesssim
2^{\frac{3jn}{2}}\int_{|z-x_{0}|>2r}\frac{|f(z)|}{|z-x_{0}|^{n}}dz.
\]
By Fubini's theorem and Lemma \ref{lemHold}, we obtain
\begin{eqnarray*}
G_{\a,2^{j}}(f_{2})(x)&\lesssim&
2^{\frac{3jn}{2}}\int_{|z-x_{0}|>2r}|f(z)|\left(
\int_{|z-x_{0}|}^{\i}\frac{dt}{t^{n+1}}\right)dz
\\
{}&\lesssim& 2^{\frac{3jn}{2}}\int_{2r}^{\i}\left(
\int_{|z-x_{0}|<t}|f(z)|\frac{dt}{t^{n+1}}\right)dz
\\
{}&\lesssim& 2^{\frac{3jn}{2}}\int_{2r}^{\i}
\|f\|_{L^{\Phi}_{w}(B(x_0,t))} \Phi^{-1}\big(w(B(x_0,t))^{-1}\big)
\frac{dt}{t}.
\end{eqnarray*}
So,
\begin{equation}\label{wueq11}
\Vert G_{\a,2^{j}}(f_{2})\Vert_{L^{\Phi}_{w}(B)}\lesssim
2^{\frac{3jn}{2}}
\int_{2r}^{\i} \|f\|_{L^{\Phi}_{w}(B(x_0,t))}
\frac{\Phi^{-1}\big(w(B(x_0,t))^{-1}\big)}
{\Phi^{-1}\big(w(B(x_0,r))^{-1}\big)}
 \frac{dt}{t}.
\end{equation}
Combining \eqref{wueq9}, \eqref{wueq10} and \eqref{wueq11}, we have
$$
\Vert G_{\a,2^{j}}(f)\Vert_{L^{\Phi}_{w}(B)}\lesssim
2^{j(\frac{3n}{2}+\a)}
\int_{2r}^{\i} \|f\|_{L^{\Phi}_{w}(B(x_0,t))}
\frac{\Phi^{-1}\big(w(B(x_0,t))^{-1}\big)}
{\Phi^{-1}\big(w(B(x_0,r))^{-1}\big)}
 \frac{dt}{t}.
$$
Consequently, we obtain
\begin{align*}
 \|G_{\a,2^{j}} f\|_{M^{\Phi,\varphi_2}_{w}(\Rn)}
 \lesssim 2^{j(\frac{3n}{2}+\a)}\sup\limits_{\substack{x_0\in\Rn\\ r>0}}
\int_{r}^{\i} \Phi^{-1}\big(w(B(x_0,t))^{-1}\big)
\frac{ \|f\|_{L^{\Phi}_{w}(B(x,t))}}{
\varphi_2(x_0,r)}\,\frac{dt}{t}. \notag
\end{align*}
Thus by Theorem \ref{thm3.2.} we have
\begin{align}\label{wueq12}
\|G_{\a,2^{j}} f\|_{M^{\Phi,\varphi_2}_{w}(\Rn)}
& \lesssim 2^{j(\frac{3n}{2}+\a)}\sup\limits_{\substack{x_0\in\Rn\\ r>0}}
\frac{\Phi^{-1}\big(w(B(x_0,r))^{-1}\big)}{\varphi_1(x_0,r)}
\|f\|_{L^{\Phi}_{w}(B(x,r))} \notag
\\
{}& = 2^{j(\frac{3n}{2}+\a)}\|f\|_{M^{\Phi,\varphi_1}_{w}(\Rn)}.
\end{align}
Since $\lambda >3+\displaystyle \frac{2\a}{n}$, by \eqref{wueq7}, \eqref{wueq8} and \eqref{wueq12}, we can conclude the proof of the theorem.

\section{Commutators of the intrinsic square functions in $M^{\Phi,\varphi}_{w}(\Rn)$}\label{s4}

We start with a characterization of the ${\rm BMO}$ norm.
\begin{lem}\label{Bmo-orlicz}
Let $0<p_0\le p_1<\i$.
Let $b\in {\rm BMO}(\Rn)$ and $\Phi$ be a Young function
which is lower type $p_0$ and upper type $p_1$.
Then
$$
\|b\|_\ast \thickapprox
\sup_{x\in\Rn, r>0}\Phi^{-1}\big(w(B(x,r))^{-1}\big)
\left\|b-b_{B(x,r)}\right\|_{L^{\Phi}_{w}(B(x,r))}.
$$
\end{lem}

\begin{proof}
By H\"{o}lder's inequality, we have
$$\|b\|_\ast \lesssim
\sup_{x\in\Rn, r>0}\Phi^{-1}\big(w(B(x,r))^{-1}\big)
\left\|b-b_{B(x,r)}\right\|_{L^{\Phi}_{w}(B(x,r))}.
$$

Now we show that
$$
\sup_{x\in\Rn, r>0}\Phi^{-1}\big(w(B(x,r))^{-1}\big)
\left\|b-b_{B(x,r)}\right\|_{L^{\Phi}_{w}(B(x,r))} \lesssim \|b\|_\ast.
$$
Without loss of generality, we may assume that $\|b\|_\ast=1$; otherwise,
we replace $b$ by $b/\|b\|_\ast$.
By the fact that $\Phi$ is lower type $p_0$ and upper type $p_1$
and \eqref{younginverse} it follows that
\begin{eqnarray*}
&&\int_{B(x,r)}
\Phi\left(\frac{|b(y)-b_{B(x,r)}|\Phi^{-1}\big(|B(x,r)|^{-1}\big)}{\|b\|_\ast}
\right)dy\\
&&=
\int_{B(x,r)}\Phi\left(|b(y)-b_{B(x,r)}|\Phi^{-1}\big(|B(x,r)|^{-1}\big)\right)
dy
\\
&&\lesssim\frac{1}{|B(x,r)|}
\int_{B(x,r)}\left[|b(y)-b_{B(x,r)}|^{p_0}+|b(y)-b_{B(x,r)}|^{p_1}\right]dy
\lesssim 1.
\end{eqnarray*}
By Lemma \ref{Kylowupp} we get the desired result.
\end{proof}

\begin{rem}
Note that a counterpart to Lemma \ref{Bmo-orlicz}
for the variable exponent Lebesgue space $L^{p(\cdot)}$ case was
obtained
in \cite{IzukiSaw}.
\end{rem}

\begin{lem}\label{lem5.1.}
Let $\a\in(0,1]$, $1<p_0\le p_1<\i$ and $b\in {\rm BMO}(\Rn)$.
Let $\Phi$ be a Young function which is lower type $p_0$ and upper type $p_1$.
Then the inequality
\begin{align*}\label{eq5.1.}
& \|[b,G_{\a}]f\|_{L^{\Phi}_{w}(B(x_0,r))}
\\
& \lesssim \frac{\|b\|_{*}}{\Phi^{-1}\big(w(B(x_0,r))^{-1}\big)}
 \int_{2r}^{\i} \Big(1+\ln \frac{t}{r}\Big)\|f\|_{L^{\Phi}_{w}(B(x_0,t))}
\Phi^{-1}\big(w(B(x_0,t))^{-1}\big) \frac{dt}{t}
\end{align*}
holds
for any ball $B(x_0,r)$ and for any $f\in L^{\Phi,\rm loc}_{w}(\Rn)$.
\end{lem}
\begin{proof}
For an arbitrary $x_0 \in\Rn$, set $B \equiv B(x_0,r)$ for the ball centered at $x_0$ and of radius $r$. Write $f=f_1+f_2$ with
$f_1\equiv f\chi_{_{2B}}$ and $f_2 \equiv f\chi_{_{\dual (2B)}}$.
We have
$
\left\|[b,G_{\a}]f \right\|_{L^{\Phi}_{w}(B)} \leq
\left\|[b,G_{\a}]f_1 \right\|_{L^{\Phi}_{w}(B)}+
\left\|[b,G_{\a}]f_2 \right\|_{L^{\Phi}_{w}(B)}
$
by the triangle inequality.
From Theorem \ref{LiNaYaZhprop2.17-1},
the boundedness of $[b,G_{\a}]$
in $L^{\Phi}_{w}(\Rn)$
it follows that
$
\|[b,G_{\a}]f_1\|_{L^{\Phi}_{w}(B)}
\leq
\|[b,G_{\a}]f_1\|_{L^{\Phi}_{w}(\Rn)}
\lesssim \|b\|_{*} \, \|f_1\|_{L^{\Phi}_{w}(\Rn)}
= \|b\|_{*} \, \|f\|_{L^{\Phi}_{w}(2B)}.
$
For $\left\|[b,G_{\a}]f_2 \right\|_{L^{\Phi}_{w}(B)}$,
we write it out in full
\begin{eqnarray*}
[b,G_{\a}]f_{2}(x) =
\left(\displaystyle \iint_{\Gamma(x)}\sup_{\phi\in C_{\a}}\left|\int_{\mathbb{R}^{n}}[b(y)-b(z)]\phi_{t}(y-z)f_{2}(z)dz\right|^{2}\frac{dydt}{t^{n+1}}\right)^{\frac{1}{2}}.
\end{eqnarray*}
We then divide it into two parts:
\begin{eqnarray*}
[b,G_{\a}]f_{2}(x)
{}&\leq&\left(\iint_{\Gamma(x)}\sup_{\phi\in C_{\a}}\left|\int_{\mathbb{R}^{n}}[b(y)-b_{B}]\phi_{t}(y-z)f_{2}(z)dz\right|^{2}\frac{dydt}{t^{n+1}}\right)^{\frac{1}{2}}
\\
{}&{}&+\left(\iint_{\Gamma(x)}\sup_{\phi\in C_{\a}}\left|\int_{\mathbb{R}^{n}}[b_{B}-b(z)]\phi_{t}(y-z)f_{2}(z)dz\right|^{2}\frac{dydt}{t^{n+1}}\right)^{\frac{1}{2}}
\\
{}&:=&{\mathfrak A}+{\mathfrak B}.
\end{eqnarray*}
First, for the quantity ${\mathfrak A}$,
we proceed as follows:
\begin{eqnarray*}
{\mathfrak A}
&=&\left(\iint_{\Gamma(x) \cap {\mathbb R}^n \times [r,\infty)}
|b(y)-b_{B}|^2
\sup_{\phi\in C_{\a}}\left|\int_{\mathbb{R}^{n}}\phi_{t}(y-z)f_{2}(z)dz\right|^{2}\frac{dydt}{t^{n+1}}\right)^{\frac{1}{2}}\\
&\lesssim &\left(\iint_{\Gamma(x) \cap {\mathbb R}^n \times [r,\infty)}
|b(y)-b_{B}|^2
\left(\frac{1}{t^n}\int_{B(x,t)}|f(z)|\,dz
\right)^2\frac{dydt}{t^{n+1}}\right)^{\frac{1}{2}}.
\end{eqnarray*}
Note that
\[
\int_{B(x,t)}|f(z)|\,dz
\lesssim
|B(x,t)|\Phi^{-1}(w(B(x,t)^{-1})
\|f\|_{L^\Phi_w(B(x,t))}.
\]
Thus,
by virtue of the embedding
$\ell^2({\mathbb N}) \hookrightarrow \ell^1({\mathbb N})$,
we obtain
\begin{align*}
{\mathfrak A}
&\lesssim
\left(\iint_{\Gamma(x) \cap {\mathbb R}^n \times [r,\infty)}
|b(y)-b_{B}|^2\Phi^{-1}(w(B(x,t)^{-1})^2
\|f\|_{L^\Phi_w(B(x,t))}{}^2
\frac{dydt}{t^{n+1}}\right)^{\frac{1}{2}}\\
&\lesssim
\left(\int_r^{\infty}\Phi^{-1}(w(B(x,t)^{-1})^2
\log\left(2+\frac{t}{r}\right)^2
\|f\|_{L^\Phi_w(B(x,t))}{}^2\,\frac{dt}{t}
\right)^{\frac{1}{2}}\\
&\lesssim
\left(
\sum_{j=1}^\infty \Phi^{-1}(w(B(x,2^jr)^{-1})^2
\log\left(2+2^j\right)^2
\|f\|_{L^\Phi_w(B(x,2^jr))}{}^2
\right)^{\frac{1}{2}}\\
&\lesssim
\sum_{j=1}^\infty \Phi^{-1}(w(B(x,2^jr)^{-1})
\log\left(2+2^j\right)
\|f\|_{L^\Phi_w(B(x,2^jr))}\\
&\lesssim
\int_r^{\infty}\Phi^{-1}(w(B(x,t)^{-1})
\log\left(2+\frac{t}{r}\right)
\|f\|_{L^\Phi_w(B(x,t))}\,\frac{dt}{t}.
\end{align*}

For the quantity ${\mathfrak B}$, since $|y-x|<t$, we have $|x-z|<2t$.
Thus, by Minkowski's inequality,
we have a pointwise estimate:
\begin{eqnarray*}
{\mathfrak B} &\leq&
\left(\displaystyle \iint_{\Gamma(x)}
\left|\int_{B(x,2t)}|b_{B}-b(z)||f_{2}(z)|dz\right|^{2}
\frac{dydt}{t^{3n+1}}\right)^{\frac{1}{2}}
\\
{}&\lesssim& \left(\displaystyle \int_{0}^{\i}
\left|\int_{B(x,2t)}|b_{B}-b(z)||f_{2}(z)|dz\right|^{2}
\frac{dt}{t^{2n+1}}\right)^{\frac{1}{2}} \\
&\lesssim&
\displaystyle \int_{\dual{B(x_0,2r)}}\frac{|b_{B}-b(z)|
|f(z)|}{|x-z|^{n}}dz.
\end{eqnarray*}
Thus, we have
\[
\|{\mathfrak B}\|_{L^{\Phi}_{w}(B)}
\lesssim \Big\|\int_{\dual (2B)}
\frac{|b(z)-b_{B}|}{|x_0-z|^{n}}|f(z)|dz\Big\|_{L^{\Phi}_{w}(B)}.
\]
Since $|z-x|\displaystyle \geq\frac{1}{2}|z-x_{0}|$, we obtain
\begin{align*}
\|{\mathfrak B}\|_{L^{\Phi}_{w}(B)}&\lesssim
\frac{1}{\Phi^{-1}\big(w(B(x_0,r))^{-1}\big)}
\int_{\dual(2B)}\frac{|b(z)-b_{B}|}{|x_0-z|^{n}}|f(z)|dz
\\
&\thickapprox \frac{1}{\Phi^{-1}\big(w(B(x_0,r))^{-1}\big)}
\int_{\dual(2B)}|b(z)-b_{B}||f(z)|\int_{|x_0-z|}^{\i}\frac{dt}{t^{n+1}}dz
\\
&\thickapprox \frac{1}{\Phi^{-1}\big(w(B(x_0,r))^{-1}\big)}
\int_{2r}^{\i}\left(\int_{2r\leq |x_0-z|\leq t}
|b(z)-b_{B}||f(z)|dz\right)\frac{dt}{t^{n+1}}
\\
&\lesssim \frac{1}{\Phi^{-1}\big(w(B(x_0,r))^{-1}\big)}
\int_{2r}^{\i}\left(\int_{B(x_0,t)}
|b(z)-b_{B}||f(z)|dz\right)\frac{dt}{t^{n+1}}.
\end{align*}
We decompose the matters by using the trinangle inequality:
\begin{align*}
\|{\mathfrak B}\|_{L^{\Phi}_{w}(B)}
&\lesssim
\frac{1}{\Phi^{-1}\big(w(B(x_0,r))^{-1}\big)} \int_{2r}^{\i}
\left(\int_{B(x_0,t)}
|b(z)-b_{B(x_0,t)}||f(z)|dz\right)\frac{dt}{t^{n+1}}
\\
&\quad +
\int_{2r}^{\i}
\frac{|b_{B}-b_{B(x_0,t)}|}{\Phi^{-1}\big(w(B(x_0,r))^{-1}\big)}
\left(\int_{B(x_0,t)} |f(z)|dz\right)\frac{dt}{t^{n+1}}
\end{align*}
Applying H\"older's inequality,
by Lemma \ref{Bmo-orlicz} and \eqref{propBMO} we get
\begin{align*}
\|{\mathfrak B}\|_{L^{\Phi}_{w}(B)}
&\lesssim
\int_{2r}^{\i}
\left\||b-b_{B(x_0,t)}|w(\cdot)^{-1}\right\|_{L^{\widetilde{\Phi}}_{w}(B)}
\frac{\|f\|_{L^{\Phi}_{w}(B(x_0,t))}dt}{t^{n+1}\Phi^{-1}\big(w(B(x_0,r))^{-1}\big)}
\\
& \quad +
\int_{2r}^{\i}|b_{B}-b_{B(x_0,t)}|
\|f\|_{L^{\Phi}_{w}(B(x_0,t))}
\frac{\Phi^{-1}\big(w(B(x_0,t))^{-1}\big)}
{\Phi^{-1}\big(w(B(x_0,r))^{-1}\big)}
\frac{dt}{t}
\\
& \lesssim
\|b\|_{*}
\int_{2r}^{\i}\Big(1+\ln \frac{t}{r}\Big)
\|f\|_{L^{\Phi}_{w}(B(x_0,t))}
\frac{\Phi^{-1}\big(w(B(x_0,t))^{-1}\big)}
{\Phi^{-1}\big(w(B(x_0,r))^{-1}\big)}
\frac{dt}{t}.
\end{align*}
Summing $\|{\mathfrak A}\|_{L^{\Phi}_{w}(B)}$ and $\|{\mathfrak B}\|_{L^{\Phi}_{w}(B)}$,
we obtain
\begin{align*}
& \|[b,G_{\a}]f_2\|_{L^{\Phi}_{w}(B)}
\\
& \lesssim \frac{\|b\|_{*}}{\Phi^{-1}\big(w(B(x_0,r))^{-1}\big)}
\int_{2r}^{\i}\Big(1+\ln \frac{t}{r}\Big) \|f\|_{L^{\Phi}_{w}(B(x_0,t))}\Phi^{-1}\big(w(B(x_0,t))^{-1}\big)\frac{dt}{t}.
\end{align*}
Finally,
\begin{align*}
& \|[b,G_{\a}]f\|_{L^{\Phi}_{w}(B)} \lesssim
\|b\|_{*}\,\|f\|_{L^{\Phi}_{w}(2B)}
\\
& + \frac{\|b\|_{*}}{\Phi^{-1}\big(w(B(x_0,r))^{-1}\big)}
\int_{2r}^{\i}\Big(1+\ln \frac{t}{r}\Big)
\|f\|_{L^{\Phi}_{w}(B(x_0,t))}\Phi^{-1}\big(w(B(x_0,t))^{-1}\big)\frac{dt}{t},
\end{align*}
and the statement of Lemma \ref{lem5.1.} follows by \eqref{ves2fd}.
\end{proof}
Finally, Theorem \ref{3.4.XcomT} follows
by Lemma \ref{lem5.1.} and Theorem \ref{thm3.2.}
in the same manner as in the proof of Theorem \ref{thm4.4.}.

\

\section{Acknowledgements}
{The research of V. Guliyev and F. Deringoz was partially supported by the grant of Ahi Evran University Scientific Research Projects (PYO.FEN.4003.13.003) and (PYO.FEN.4003-2.13.007).
We thank the referee for he/her valuable comments to the paper.}

\

\

$^{a}$ Department of Mathematics, Ahi Evran University, Kirsehir, Turkey

$^{b}$ Institute of Mathematics and Mechanics, Baku, Azerbaijan

{\it E-mail address}: vagif@guliyev.com

\

$^{c}$ Baku State University, Baku, AZ 1148, Azerbaijan

{\it E-mail address}: mehriban\_omarova@yahoo.com

\

$^{d}$ Department of Mathematics and Information Science, Tokyo Metropolitan University

{\it E-mail address}: yoshihiro-sawano@celery.ocn.ne.jp

\end{document}